\documentclass[12pt]{amsart}
\usepackage{geometry} 
\usepackage{amssymb}
\usepackage{amscd}
\usepackage{amsmath}
\usepackage[all,cmtip]{xy}
\usepackage{mathrsfs}
\usepackage{amsthm}
\usepackage[pdfpagelabels]{hyperref}
\usepackage{multirow}

\newtheorem{thm}{Theorem}[section]

\newtheorem{lem}[thm]{Lemma}
\newtheorem{prop}[thm]{Proposition}
\newtheorem{conj}[thm]{Conjecture}

\theoremstyle{definition}

\theoremstyle{remark}
\newtheorem{remark}[thm]{Remark}

\begin{document}

\title[Examples of finitely determined map-germs of corank 3]{EXAMPLES of FINITELY DETERMINED MAP-GERMS OF \\ CORANK 3 SUPPORTING MOND'S  $\mu \geq \tau$-TYPE CONJECTURE}

\author{Ay{\c s}e Sharland}


\email{aysealtintas@gmail.com}




\maketitle
\begin{flushright}\begin{small}\textit{Dedicated to my parents}\end{small}\end{flushright}
\section{Introduction}

A famous $\mu \geq \tau$-type conjecture by D. Mond states that for a finitely $\mathcal{A}$-determined map-germ from $\mathbb{C}^n$ to $\mathbb{C}^{n+1}$, provided $(n,n+1)$ is in the range of Mather's nice dimensions, 
\begin{equation}\label{conj}\mu_I\geq \mathcal{A}_e\textnormal{-codimension}, \end{equation}
and with equality if the map-germ is weighted homogeneous. 
The conjecture was proven for $n=1$ by Mond (\cite{mond-bent}) and for $n=2$ by 
Pellikaan and de Jong (unpublished), de Jong and Straten (\cite{jong-straten}) and D. Mond (\cite{mond89}). It is still open for $n\geq 3$. Several examples supporting the conjecture were given in the case of map-germs of corank 1 (\cite{houston-kirk}) and corank 2 (\cite{altintas2}). It was believed by some that it would only hold for map-germs of corank $\leq 2$. In this article, we provide examples of finitely determined map-germs of corank 3 defined from $\mathbb{C}^3$ to $\mathbb{C}^4$ which satisfy the conjecture (Section \ref{sect-ex}). These are the first examples in the literature known to the author.

\section{Terminology and Notations}

\subsection{Finite determinacy}

Our terminology is standard, but the details can be found in \cite{wall} or \cite{martinet}. We denote the space of holomorphic map-germs $f\colon (\mathbb{C}^n,0)\rightarrow (\mathbb{C}^p,0)$ by $\mathcal{E}_{n,p}^0$. The group $\mathcal{A}:=\textnormal{Diff}(\mathbb{C}^n,0)\times \textnormal{Diff}(\mathbb{C}^p,0)$ of local diffeomorphisms  acts on $\mathcal{E}_{n,p}^0$  by 
\[(\phi,\psi)\cdot f\mapsto \psi \circ f \circ \phi^{-1}\]
 for all $(\phi,\psi)\in\mathcal{A}$.
We say that $f,g\in\mathcal{E}_{n,p}^0$ are $\mathcal{A}$-\textit{equivalent} if $g\in \mathcal{A}\cdot f$. A map-germ $f\in \mathcal{E}_{n,p}^0$ is $\ell$-$\mathcal{A}$-\textit{determined} if every map-germ $g\in\mathcal{E}_{n,p}^0$ with the same $\ell$-jet (at $0$) as $f$ is $\mathcal{A}$-equivalent to $f$. Furthermore, $f$ is  \textit{finitely $\mathcal{A}$-determined} (or $\mathcal{A}$-\textit{finite}) if it is $\ell$-$\mathcal{A}$-determined for some $\ell<\infty$. A map-germ is $\mathcal{A}$-\textit{stable} if any of its unfoldings is $\mathcal{A}$-equivalent to the trivial unfolding $f\times 1$. By fundamental results of Mather, finite determinacy is equivalent to the finite dimensionality of the normal space 
\begin{equation}\label{normalsp}N\mathcal{A}_ef:=\frac{f^*(\Theta_{\mathbb{C}^p,0})}{tf(\Theta_{\mathbb{C}^n,0})+f^{-1}(\Theta_{\mathbb{C}^p,0})},\end{equation} and thus (if $f$ is not stable) to $0\in\mathbb{C}^p$ being an isolated point of instability of $f$. We set  $\mathcal{A}_e\textnormal{-codim}(f):=\textnormal{dim}_{\mathbb{C}}N\mathcal{A}_ef$. 

The corank of a map-germ $f\in \mathcal{E}_{n,p}^0$ with $n\leq p$ is defined to be 
\[ \textnormal{corank }f=n-\textnormal{rank } \textnormal{d}f(0).\]

\begin{remark}\textit{a.} There are a few methods to calculate $\mathcal{A}_e$-codimension for a given finite map-germ and each may have certain disadvantages. Calculating it directly from the definition is not always practical since one has to do it by hand -- the normal space is not an $\mathcal{O}_{\mathbb{C}^n,0}$-module and that makes it difficult to write it into a computer algorithm.\newline \textit{b.} Alternatively, one can use J. Damon's theory where he relates $\mathcal{A}$-equivalence with $\mathcal{K}_V$ equivalence and shows that for a finitely $\mathcal{A}$-determined map-germ $f\in \mathcal{E}_{n,p}^0$, the normal spaces with respect to $\mathcal{A}$ and $\mathcal{K}_V$ are isomorphic:
\begin{equation}\label{eq-akv} N\mathcal{A}_ef \cong N\mathcal{K}_{V,e}g
\end{equation}
where $V$ is the image of a stable unfolding $F$ of $f$ and $g$ is the pull-back map from $(\mathbb{C}^p,0)$ to the target space of $F$ (\cite{damon89}, see also \cite[Theorem 8.1]{mond-diff}). The right hand side of (\ref{eq-akv}) can easily be adapted to a computer algebra program (see \cite{altintas2} for examples). However, this procedure requires a long time to complete when the number of parameters for a stable unfolding is too big, as for the examples we study in this article. \end{remark}

Here, we will use the following proposition which provides a shorter and much faster algorithm to calculate $\mathcal{A}_e$-codimension. 

\begin{prop}[Proposition 2.1, \cite{mond89}]\label{prop-acod} Let $h$ be a defining equation of the image $(X,0)$ of the finitely $\mathcal{A}$-determined  map-germ $f\in\mathcal{E}_{n,n+1}^0$. Then the evaluation on $h$ defines an isomorphism of $\mathcal{O}_{\mathbb{C}^{n+1},0}$-modules
\begin{equation}\label{eq-acod} 
N\mathcal{A}_e f \cong \frac{J_h\mathcal{O}_{\mathbb{C}^n,0}}{J_h\mathcal{O}_{X,0}}.
\end{equation}
\end{prop}

\begin{remark}\label{rem-acod} We understand from the proof of Proposition \ref{prop-acod} that it is sufficient for $f$ to have a ramification locus of codimension 2 to have the isomorphism in (\ref{eq-acod}).\end{remark}

 In what follows, we will denote the right hand side of (\ref{eq-acod}) by $N_f$.

\subsection{Topology of the image}

If $f\in\mathcal{E}_{n,p}^0$ is finitely $\mathcal{A}$-determined then it has an isolated instability at the origin (\cite[p. 241]{mather73},\cite{gaffney}). Moreover, if $(n,n+1)$ are nice dimensions then the image of a stabilisation of $f$ has the homotopy type of wedge of $n$-spheres (\cite[Theorem 1.4]{mond89}). The number of $n$-spheres in the wedge is called the image Milnor number and denoted by $\mu_I$.

\begin{remark} \textit{a.} For map-germs of corank 1 in $\mathcal{E}_{n,n+1}^0$,  Goryunov and Mond gave a method to calculate the cohomology of the image of a stable perturbation using alternating cohomology groups of multiple point spaces and that provides a formula for the image Milnor number (\cite{goryunov-mond}). In \cite{houston-local}, K. Houston showed that the same formula holds for stable perturbations of map-germs of any corank. See \cite{mond-virtual} for detailed calculations of $\mu_I$ for corank 2 map-germs based on these ideas.
\newline \textit{b.} For weighted homogeneous map-germs of any corank in $\mathcal{E}_{2,3}^0$, Mond has an ingenous formula for $\mu_I$ given in terms of weights and degrees (\cite{mond-number}). In \cite{ohmoto}, T. Ohmoto improved it to weighted homogenous map-germs in $\mathcal{E}_{3,4}^0$ using characteristic classes and Thom polynomials. That is the formula we will use for our examples in this article.
\end{remark}

Clearly, proving the conjecture will provide an alternative method to calculate the image Milnor Number for weighted homogeneous map-germs of any corank. One of the ideas about how to attack the conjecture is based on the relation between $\mathcal{A}_e$-equivalance and Damon's $\mathcal{K}_H$-equivalence: The conjecture holds if and only if a particular relative normal space $N\mathcal{K}_{H,e/\mathbb{C}}G$ is a Cohen-Macaulay module (\cite{altintas2}, \cite{mond15}). Recently, J. F. Bobadilla,  J.J. Nu{\~n}o and G. Pe{\~n}afort proved that it is also equivalent to showing that a \textit{jacobian} module (a relative version of the module $M(f)$ mentioned in Remark \ref{rem-bnp}) has the Cohen-Macaulay property (\cite{B-N-P}).




\section{Examples}\label{sect-ex}

Before we present our examples, we restate the definition for $N_f$ that will help us putting our calculations into a computer algorithm.

\begin{prop}\label{prop-sing} Let $f\in\mathcal{E}_{n,n+1}^0$ be a finite map-germ and let $X$ be its image, defined by an ideal $h\in \mathcal{O}_{\mathbb{C}^{n+1},0}$. Assume that $f$ is a weighted homogeneous map-germ. Then 
\[N_f= \frac{(f^*)^{-1}\left((f^*J_h)\mathcal{O}_{\mathbb{C}^n,0}\right)}{J_h\mathcal{O}_{\mathbb{C}^{n+1},0}}.\] If, in addition, the ramification locus of $f$ has codimension 2 then 
\[\textnormal{dim}_\mathbb{C} N_f=\mathcal{A}_e\textnormal{-codim}(f).\]
\end{prop}

\begin{proof} Our argument is based on exploiting the $\mathcal{O}_{\mathbb{C}^{n+1},0}$-module structure of $N_f$. The definition of $N_f$ in Proposition \ref{prop-acod} reads as 
\begin{equation}\label{eq-Nf1} N_f=\frac{(f^*J_h)\mathcal{O}_{\mathbb{C}^n,0}}{J_h\mathcal{O}_{X,0}}.
\end{equation}
Let  $\mathcal{C}$ be the conductor ideal  of $\mathcal{O}_{\mathbb{C}^n,0}$ in  $\mathcal{O}_{X,0}$, that is,
\[ \mathcal{C}=\left\{r\in  \mathcal{O}_{X,0} \mid r\cdot  \mathcal{O}_{\mathbb{C}^n,0} \subset  \mathcal{O}_{X,0}\right \}.
\]  We have the following inclusion of the ideals $ J_h \subseteq \textnormal{Fitt}_1(f_*\mathcal{O}_{\mathbb{C}^n,0}) \subseteq \mathcal{C} $
(see \cite[p.121]{mond-pellikaan} for the second inclusion). Notice that since $J_h\subseteq \mathcal{C}$ and  
\[\sum_{i=1}^{n+1}\alpha_i\frac{\partial h}{\partial Y_i} \in \mathcal{O}_{X,0} \]
for any $\alpha_i\in\mathcal{O}_{\mathbb{C}^n,0}$, $(f^*J_h)\mathcal{O}_{\mathbb{C}^n,0}$ is an ideal both in $\mathcal{O}_{\mathbb{C}^n,0}$ and in $\mathcal{O}_{X,0}$. Hence, the map
\[\bar{f}^*\colon \mathcal{O}_{X,0} \rightarrow \mathcal{O}_{\mathbb{C}^n,0}\]
contains $(f^*J_h)\mathcal{O}_{\mathbb{C}^n,0}$ in its image. So, instead of (\ref{eq-Nf1}), we can take
\[N_f=\frac{(\bar{f}^*)^{-1}\left((f^*J_h)\mathcal{O}_{\mathbb{C}^n,0}\right)}{J_h\mathcal{O}_{X,0}}.\]
As $f$ is weighted homogeneous, we have $h\in J_h\mathcal{O}_{\mathbb{C}^{n+1},0}$. Therefore,
\[\frac{(\bar{f}^*)^{-1}\left((f^*J_h)\mathcal{O}_{\mathbb{C}^n,0}\right)}{J_h\mathcal{O}_{X,0}}=\frac{(f^*)^{-1}\left((f^*J_h)\mathcal{O}_{\mathbb{C}^n,0}\right)}{J_h\mathcal{O}_{\mathbb{C}^{n+1},0}}.\]
Finally, the second part of the statement follows from Remark \ref{rem-acod}.
\end{proof}

\begin{remark}\label{rem-bnp} The same result, but with a different approach, can also be found in \cite{B-N-P}. There, the authors study the kernel $M(f)$ of the epimorphism 
\[\frac{\mathcal{C}}{J_h} \rightarrow \frac{\mathcal{C}}{J_h\mathcal{O}_{\mathbb{C}^n,0}}.\] They show that  
\[M(f)=\frac{(f^*)^{-1}\left((f^*J_h)\mathcal{O}_{\mathbb{C}^n,0}\right)}{J_h\mathcal{O}_{\mathbb{C}^{n+1},0}}\] and that, for weighted homogeneous map-germs, $M(f)=N_f$ (\cite[Proposition 3.3, Proposition 5.1]{B-N-P}).


\end{remark}

\begin{prop}\label{prop-ex1} The map-germ
\begin{eqnarray*} f_1\colon (\mathbb{C}^3,0)&\rightarrow &(\mathbb{C}^4,0) \\
(x,y,z)&\mapsto & (y^2+xz,\ x^5+yz+xy^2,\ x^6+y^3+z^2,\ x^7+x^4z+xz^2+y^2z)
\end{eqnarray*}
has $\mathcal{A}_e$-codimension equal to 18967. It is weighted homogeneous with weights $(1,2,3)$ and degrees $(4,5,6,7)$, of corank 3 and satisfy the Mond conjecture.
\end{prop}

\begin{proof} Firstly, we have 
\begin{equation}\label{rf} \textnormal{d}f= \left[ \begin{array}{ccc}  z &        2y &   x \\       
5x^4+y^2 &     z+2xy & y \\       
6x^5 &        3y^2 & 2z \\      
7x^6+4x^3z+z^2 & 2yz &  x^4+2xz+y^2
\end{array}\right] \end{equation}
and $\textnormal{d}f(0)$ is the zero matrix. Hence $f_1$ is of corank 3. The ramification locus is defined by $3\times 3$-minors of $\textnormal{d}f$. Its codimension is equal to 2. We check that by the following \textsc{Singular} (\cite{singular}) code.
\begin{verbatim}
ring s=0,(x,y,z),(wp(1,2,3));
ideal f=y2+xz,x5+yz+xy2,x6+y3+z2,x7+x4z+xz2+y2z;
matrix df=jacob(f);
ideal rf=std(minor(df,3));
dim(rf);
//->1.
\end{verbatim}
Hence, we can apply Proposition \ref{prop-sing} to calculate $\mathcal{A}_e$-codimension of $f_1$, i.e. the vector space dimension of $N_f$. We run the following code to find that.
\vskip8pt 
\noindent\verb|ring t=31991,(X,Y,Z,W),(wp(4,5,6,7));| \quad  $//$\ the target of $f_1$ \\
\verb|ring s=31991,(x,y,z),(wp(1,2,3));| \quad  $//$\ the domain of $f_1$ \\
\verb|map f1=t,y2+xz,x5+yz+xy2,x6+y3+z2,x7+x4z+xz2+y2z;| \\
\verb|ideal p=0;| \\
\verb|setring t;| \\
\verb|ideal h=preimage(s,f1,p);| \quad  $//$\ the ideal defining the image of $f_1$ \\
\verb|ideal jh=jacob(h);| \\
\verb|setring s;| \\
\verb|ideal fjh=f1(jh);  | \quad $//$ $f_1^*J_h$ \\
\verb|setring t;| \\
\verb|ideal ffjh=preimage(s,f1,fjh);| \quad $//$ $(f_1^*)^{-1}f_1^*J_h$ \\
\verb|def N=modulo(ffjh,jh);| \\
\verb|vdim(std(N));| \\
\verb|//->18967.|
\vskip8pt 

On the other hand, Ohmoto's formula (\cite{ohmoto}) for $\mu_I$ for weighted homogenous map-germs in $\mathcal{E}_{3,4}^0$ also gives 18967. Therefore, $f_1$ satisfies the conjecture. \end{proof}

\begin{remark} We carry out our calculations over characteristic 31991 since the computer struggles to give an output over characteristic 0. This choice does not effect the outcome of the code. 
\end{remark}

Similarly, we can confirm the following examples.
\begin{prop}\label{prop-ex2} The map-germ
\begin{eqnarray*} f_2\colon (\mathbb{C}^3,0)&\rightarrow &(\mathbb{C}^4,0) \\
(x,y,z)&\mapsto & (y^2+xz,\ x^5+yz+xy^2,\ x^6+y^3+z^2,\ x^9+x^6z+z^3+y^3z)
\end{eqnarray*}
has $\mathcal{A}_e$-codimension equal to 41244. It is weighted homogeneous with weights $(1,2,3)$ and degrees $(4,5,6,9)$, of corank 3 and satisfy the Mond conjecture.
\end{prop}

\begin{prop}\label{prop-ex3} The map-germ
\begin{eqnarray*} f_3\colon (\mathbb{C}^3,0)&\rightarrow &(\mathbb{C}^4,0) \\
(x,y,z)&\mapsto & (y^2+xz,\ x^5+yz+xy^2,\ x^6+y^3+z^2,\ x^{13}+x^{10}z+xz^4+y^5z)
\end{eqnarray*}
has $\mathcal{A}_e$-codimension equal to 127295. It is weighted homogeneous with weights $(1,2,3)$ and degrees $(4,5,6,13)$, of corank 3 and satisfy the Mond conjecture.
\end{prop}

\begin{remark} It might seem like these three map-germs are parts of a series of map-germs. Ohmoto's formula for the weights $(1,2,3)$ and degrees $(4,5,6,2k+1)$ gives the integer
\[\mu_I=487k^3+576k^2+197k+18+\frac{4k^3+3k^2+5k}{6}\]
for all $k\geq 1$. However, the following map-germs are not finitely $\mathcal{A}$-determined. 
\begin{eqnarray*} (x,y,z)&\mapsto & (y^2+xz,\ x^5+yz+xy^2,\ x^6+y^3+z^2,\ x^{11}+x^{8}z+x^2z^3+y^4z), \\
(x,y,z) &\mapsto & (y^2+xz,\ x^5+yz+xy^2,\ x^6+y^3+z^2,\ x^{15}+x^{12}z+z^5+y^6z), \\
(x,y,z)&\mapsto & (y^2+xz,\ x^5+yz+xy^2,\ x^6+y^3+z^2,\ x^{17}+x^{14}z+x^2z^5+y^7z). \end{eqnarray*}
Of course, this does not prove that there are not any finitely $\mathcal{A}$-determined map-germs in $\mathcal{E}_{3,4}^0$ with weights $(1,2,3)$ and degrees $(4,5,6,11)$, $(4,5,6,15)$ or $(4,5,6,17)$. 
\end{remark}

\section{Other invariants}

In this section we talk about some invariants for the map-germ $f_1$ introduced in Proposition \ref{prop-ex1}. Let us put $f=f_1$ to simplify our notation. The  \textit{multiplicity} of $f$ is
\[\textnormal{qf}(f):=\textnormal{dim}_\mathbb{C}\frac{\mathcal{O}_{\mathbb{C}^3,0}}{f^*\mathfrak{m}_{\mathbb{C}^4,0}}=17.\] 
There exists a presentation of $f_*\mathcal{O}_{\mathbb{C}^3,0}$ of the form
 \begin{equation*}  0\rightarrow 
\mathcal{O}_{\mathbb{C}^{4},0}^{17} \xrightarrow{\textnormal{    }\Lambda\textnormal{    }} \mathcal{O}_{\mathbb{C}
^{4},0}^{17} \xrightarrow{\epsilon} f_*\mathcal{O}_{\mathbb{C}^{3},0}\rightarrow 0 
\end{equation*} 
where $\Lambda$ is a symmetric $17\times 17$-matrix and $\epsilon$ can be chosen as
\[ \epsilon=\left[ \begin{array}{ccccccccccccccccc} 1 & x & y & z & x^2 & xy & xz & yz & z^2 & x^3 & x^2y & xyz & x^2z & xz^2 & x^4 & x^3y & x^4y  \end{array} \right]. \]
The $k$'th multiple point space $M_k$ on the image is defined by the $(k-1)$'st Fitting ideal $\textnormal{Fitt}_{k-1}(f_*\mathcal{O}_{\mathbb{C}^3,0})$ of  $f_*\mathcal{O}_{\mathbb{C}^3,0}$, i.e. the ideal of $(17-(k-1))\times(17-(k-1))$-minors of $\Lambda$. For example, $M_1$ is the image, $M_2$ is the double point space, ... etc. Let $D^k_1(f)$ be the $k$'th multiple point space on the domain with an analytic structure given by $f^*\textnormal{Fitt}_{k-1}(f_*\mathcal{O}_{\mathbb{C}^3,0})$. So, set theoretically, $D^k_1(f)=f^{-1}(M_k)$.

Since $f$ is finitely $\mathcal{A}$-determined, the multiple point spaces are dimensionally correct, that is, 
\[ \textnormal{dim}_\mathbb{C} D^k_1(f)=4-k \]
for $k=1,2,3,4$. Moreover, no finite map-germ in $\mathcal{E}_{3,4}^0$ admits any \textit{genuine} 5-tuple point.\footnote{By a genuine $k$-tuple point, we refer to a point which splits into $k$ distinct points under a stable perturbation. } Hence, we are interested in $D^k_1(f)$ only for $k=2,3,4$. 

\subsection{Triple points}
\begin{lem} $\mu\left(D^3_1(f)\right)=168341$.
\end{lem} 
\begin{proof} We use Greuel's formula for weighted homogeneous space curves given by 
\[\mu =\tau -t+1\]
where $t$ is the Cohen-Macaulay type, i.e. the second Betti number, of the singularity (\cite{greuel-deform}). 

A direct calculation shows that $D^3_1(f)$ is a Cohen-Macaulay space curve with an isolated singularity at the origin. Moreover, it can be defined by $16\times 16$-minors of a $17\times 16$-matrix over $\mathcal{O}_{\mathbb{C}^3,0}$.\footnote{In fact, by the Hilbert-Burch theorem, any Cohen-Macaulay variety of codimension 2 can be defined by $r\times r$-minors of an $(r+1)\times r$-matrix (see, for example, \cite[Theorem 20.15]{eisenbud}).} So, we can use Fr{\"u}hbis-Kr{\"u}ger's theory (\cite{fruhbis-kruger-sp}) on matrix singularities to calculate $T^1$ of $D^3_1(f)$. We run the following code on \textsc{Singular}.

\vskip8pt 
\noindent\verb|LIB "spcurve.lib"; | \\
\verb|//| the ideal of $D^3_1(f)$ is \verb|d31| \\ 
\verb|matrix m31=syz(d31);| \quad // a matrix representation of \verb|d31|\\
\verb|list t31=matrixT1(m31,3);| \\
\verb|vdim(std(t31[2]));| \quad // the Tjurina number of $D^3_1(f)$ \\
\verb|//->168356| \\
\verb|CMtype(d31);|  \quad // Cohen-Macaulay type of $D^3_1(f)$\\
\verb|//-> 16| 
\vskip8pt 
\noindent Therefore, 
\[\mu(D^3_1(f))=\tau-t+1=168356-16+1=168341.\]
\end{proof}

The ramification locus $R_{f}$ of $f_1$ is also a space curve with an isolated singularity at the origin. Its matrix is given by (\ref{rf}). So, its Cohen-Macaulay type is $3$. We also find that  $\tau(R_f)=127$. Hence,

\begin{lem} $\mu(R_{f})=125$.
\end{lem}

\subsection{Quadruple points}
Finding the number of quadruple points of a stable perturbation of $f$ requires a little bit of work. The analytic structure of $M_4$ only gives us information about the geometrical picture of quadruple points. Whether the vector space dimension of the 3rd Fitting ideal counts the number of quadruple points of a stable perturbation of a map-germ in $\mathcal{E}_{3,4}^0$ is still an open question. For a proof, we need to show that the module
\[M:= \frac{\mathcal{O}_{\mathbb{C}^4\times \mathbb{C}^d,0}}{\textnormal{Fitt}_3(F_*\mathcal{O}_{\mathbb{C}^3\times\mathbb{C}^d,0})}\] 
 satisfy the principle of conservation (\cite[Theorem 6.4.7]{jong-pfister}), where $F\in \mathcal{E}_{3+d,4+d}^0$  is a stable unfolding of $f$. That is, the stalk of $M$ at $0$ is a free $\mathcal{O}_{\mathbb{C}^d,0}$-module of finite rank. 
However, it is a huge task for a computer to conclude such calculation for our example -- a stable unfolding of $f$ requires a minimum of 40 parameters. At the moment, we can only conjecture the number of quadruple points. 

\begin{conj} The number of quadruple points is $8970$.
\end{conj}

\subsection{Double points}
For a stable perturbation $f_t$ of $f$, $D^2_1(f_t)$ has the homotopy type of a wedge of $2$-spheres (see \cite[Remark 1.1 (2)]{mond-virtual}). We would also like to calculate the number of spheres in the wedge using the methods explained in \cite{mond-virtual}. However, due to computer memory restrictions, we have to leave this question to another study.

\section*{Acknowledgments}
The author would like to thank David Mond for his comments on the proof of Proposition \ref{prop-sing}.

\bibliographystyle{amsplain}
\bibliography{a-reference}
\end{document}